\documentclass{amsart} % please use amsart at 11pt
\usepackage{amssymb,latexsym}
\usepackage[OT2,T1]{fontenc}

\theoremstyle{plain}
\newtheorem{theorem}{Theorem}[section]
\newtheorem{lemma}[theorem]{Lemma}
\newtheorem{corollary}[theorem]{Corollary}
\newtheorem{proposition}[theorem]{Proposition}

\theoremstyle{definition}
\newtheorem{definition}{Definition}[section]
\newtheorem{example}{Example}[section]

\theoremstyle{remark}
\newtheorem{remark}{Remark}[section]

\newcommand{\N}{\mathbf N}

\newcommand{\f}{\varphi}
\newcommand{\e}{\varepsilon}

\renewcommand{\a}{\alpha}

\numberwithin{equation}{section} % to get equations numbered
\begin{document}
\title[Compact and "compact" operators]{Compact and "compact" operators on the standard Hilbert module over a $W^*$ algebra} % please provide
                                % an abbreviated title

\author{Dragoljub J. Ke\v cki\' c}
\address{University of Belgrade\\ Faculty of Mathematics\\ Student\/ski trg 16-18\\ 11000 Beograd\\ Serbia}

\email{keckic@matf.bg.ac.rs}

\author{Zlatko Lazovi\'c}

\address{University of Belgrade\\ Faculty of Mathematics\\ Student\/ski trg 16-18\\ 11000 Beograd\\ Serbia}

\email{zlatkol@matf.bg.ac.rs}

\thanks{The authors was supported in part by the Ministry of education and science, Republic of Serbia, Grant \#174034.}

\begin{abstract}
We construct a topology on the standard Hilbert module $l^2(\mathcal A)$ over a unital $W^*$-algebra
$\mathcal A$ such that any "compact" operator, (i.e.\ any operator in the norm closure of the linear span of
the operators of the form $x\mapsto\left<y,x\right>z$) maps bounded sets into totally bounded sets.
\end{abstract}

%\dedicatory{This paper is dedicated to Professor X on his 125th birthday.}



\maketitle

\section{Introduction}

Given a unital $W^*$-algebra $\mathcal A$ we consider the standard Hilbert module denoted by $l^2(\mathcal
A)$, (the notation $\mathcal H_{\mathcal A}$ is also widespread)
$$l^2(\mathcal A)=\{x=(\xi_1,\xi_2,\dots)\;|\;\xi_j\in\mathcal A,\sum_{j=1}^{+\infty}\xi^*_j\xi_j\mbox{
converges in the norm topology}\},$$
equipped with the $\mathcal A$-valued inner product
$$l^2(\mathcal A)\times l^2(\mathcal A)\ni(x,y)\mapsto\sum_{j=1}^{+\infty}\xi_j^*\eta_j\in\mathcal A,\qquad x=(\xi_1,\xi_2,\dots),\quad y=(\eta_1,\eta_2,\dots).$$

Since an arbitrary $\mathcal A$-linear bounded operator on $l^2(\mathcal A)$ does not need to have an
adjoint, the natural algebra of operators is $B^a(l^2(\mathcal A))$ - the algebra of all $\mathcal A$-linear
bounded operators on $l^2(\mathcal A)$ having an adjoint. It is known that $B^a(l^2(\mathcal A))$ is a
$C^*$-algebra and also that it is a $W^*$-algebra whenever $\mathcal A$ is of that kind.

Among all operators in $B^a(l^2(\mathcal A))$, those that belong to the linear span of the operators of the
form $x\mapsto\Theta_{y,z}(x)=z\left<y,x\right>$ ($y$, $z\in l^2(\mathcal A)$) are called {\em finite rank
operators}. The norm closure of finite rank operators is known as the algebra of all "compact" operators. The
quotation marks are usually written in order to emphasize the fact that "compact" operators does not maps
bounded sets into relatively compact sets, as it is the case in the framework of Hilbert (and also Banach)
spaces, though they share many properties of proper compact operators on a Hilbert space, \cite{Manuilov1},
\cite{Manuilov2}

For general literature concerning Hilbert modules over more general $C^*$ algebras, including the standard
Hilbert module, the reader is referred to \cite{Lance} or \cite{MT}.

The aim of this note is to introduce a locally convex topology on $l^2(\mathcal A)$, where $\mathcal A$ is a
unital $W^*$-algebra, such that any "compact" operator maps bounded sets (in the norm) into totally bounded
in the introduced topology. In a very special case, where $\mathcal A\cong B(H)$ the algebra of all bounded
operators on a Hilbert space, the converse is also true. Namely, any operator $T\in B^a(l^2(\mathcal A))$
that maps bounded into totally bounded sets is "compact". Therefore, speaking freely, we can omit the
quotation marks.

\section{Preliminaries}

Let us recall some basic definitions and facts concerning uniform spaces. For more details see
\cite{Bourbaki} or \cite{Kelly}.

Uniform spaces are those topological spaces in which one can deal with notions such as Cauchy sequence,
Cauchy net or uniform continuity. Although it is usual to define them as spaces endowed with a family of sets
in $X\times X$ given as some kind of neighborhoods of the diagonal, so called {\em entourages}, for our
purpose it is more convenient to give an equivalent definition, via a family of semimetrics.

\begin{definition} A nonempty set endowed with a family of semimetrics, functions $d_\a:X\times
X\to[0,+\infty)$ satisfying ($i$) $d_\a(x,y)\ge0$; ($ii$) $d_\a(x,y)=d_\a(y,x)$; ($iii$) $d_\a(x,z)\le
d_\a(x,y)+d_\a(y,z)$ is called a uniform space.

All $d_\a$ are metrics except they do not distinguish points, i.e.\ there might be $d_\a(x,y)=0$ for some
$x\neq y$. However it is provided that for all $x\neq y$ there is an $\a$ such that $d_\a(x,y)>0$.
\end{definition}

The family of sets $B_{d_\a}(x;\e)=\{y\in X\;|\;d_\a(x,y)<\e\}$ makes a basis for some topology. It is well
known that a topological space $X$ is a uniform space if and only if it is completely regular.

Let $X$ be a uniform space. For a net $x_i\in X$ we say that it is a Cauchy net if it is a Cauchy net with
respect to all $d_\a$, i.e.\ if for all $\a$ and for all $\e>0$ there is $i_0$ such that for all $i$, $j>i_0$
we have $d_\a(x_i,x_j)<\e$. The notion of a complete uniform space is defined in an obvious way.

A set $A\subseteq X$ is called totally bounded, if for all $\e>0$ and all $\a$ there is a finite set $c_1$,
$c_2$, $\dots$, $c_m\in X$ such that $B_\a(c_j;\e)=\{y\in X\;|\;d_\a(c_j,y)<\e\}$ cover $A$. It is well known
that any relatively compact set is totally bounded, and that the converse is true provided that $X$ is
complete.

If $X$ is not complete then there are totally bounded sets that are not relatively compact, for instance,
$\mathbb Q\cap[0,1]$ as a subset of $\mathbb Q$. (See also \cite[Remark 4.2.2]{Bourbaki})

Any locally convex topological vector space is a uniform space. Indeed, there is a family of seminorms
generating its topology. This family can be obtained by Minkowski functionals of basic neighborhoods of zero.
And an arbitrary seminorm define a semimetric in a natural way. Conversely, any family of seminorms that
distinguishes points leads to a locally convex Hausdorff topological vector space. Hence a family of
seminorms allows to deal with notions: totally bounded set, complete space, Cauchy net, etc.

\section{Topology}

For an arbitrary Hilbert $W^*$-module $\mathcal M$, Paschke \cite{Paschke}, \cite{Paschke2} in his initial
works on Hilbert $C^*$-modules and Frank \cite{Frank90}, introduced two topologies, $\tau_1$ and $\tau_2$,
the first of them generated by functionals $x\mapsto\f(\left<y,x\right>)$, $y\in\mathcal M$, $\f$ normal
state, and the second by seminorms $p(x)=\f(\left<x,x\right>)^{1/2}$, $\f$ normal state. Frank proved that
$\mathcal M$ is self-dual if and only if the unit ball in $\mathcal M$ is complete in $\tau_1$ (and this is
equivalent to the completeness in $\tau_2$). Therefore, if $\mathcal M=l^2(\mathcal A)$ is a standard Hilbert
module, it is not complete neither in $\tau_1$ nor in $\tau_2$, since $l^2(\mathcal A)$ is never self-dual,
except in the case where $\mathcal A$ is finite dimensional algebra. Since obviously $\tau_1\subset\tau_2$,
we shall refer to these topologies as to {\em weak PF} and {\em strong PF} topologies.

However, we need a topology which is between weak and strong PF topology. Namely, on a standard Hilbert
module $l^2(\mathcal A)$ where $\mathcal A$ is a unital $W^*$ algebra we define a locally convex topology
$\tau$ by the family of seminorms
\begin{equation}\label{seminorms}
p_{\f,y}(x)=\sqrt{\sum_{j=1}^{+\infty}|\f(\eta_j^*\xi_j)|^2},
\end{equation}
where $\f$ is a normal state, and $y=(\eta_1,\eta_2,\dots)$ is a sequence of elements in $\mathcal A$ such
that
\begin{equation}\label{admissible_sequence}
\sup_{j\ge1}\f(\eta_j^*\eta_j)=1.
\end{equation}

\begin{proposition}\label{ComparisonTopologies}
 Seminorms (\ref{seminorms}) are well defined. Also $\tau_1\subset\tau\subset\tau_2$.
\end{proposition}

\begin{proof} Since $(\xi,\eta)\mapsto\f(\eta^*\xi)$ is a semi inner product, we have
$|\f(\eta_j^*\xi_j)|^2\le\f(\xi_j^*\xi_j)\f(\eta_j^*\eta_j)$. By this, and by (\ref{admissible_sequence}) we
have
\begin{equation}\label{tau<tau_2}
p_{\f,y}(x)^2=\sum_{j=1}^{+\infty}|\f(\eta_j^*\xi_j)|^2\le\sum_{j=1}^{+\infty}\f(\xi_j^*\xi_j)\f(\eta_j^*\eta_j)\le\sum_{j=1}^{+\infty}\f(\xi_j^*\xi_j)=\f(\left<x,x\right>).
\end{equation}
This proves that seminorms (\ref{seminorms}) are well defined, and also that $\tau\subset\tau_2$.

To prove $\tau_1\subset\tau$, pick $y\in l^2(\mathcal A)$, $y=(\eta_1,\eta_2,\dots)$. The sequence $\zeta_j$
given by $\zeta_j=\eta_j/\f(\eta_j^*\eta_j)^{1/2}$ if $\f(\eta_j^*\eta_j)\neq0$, and $\zeta_j=0$ otherwise
obviously fulfils (\ref{admissible_sequence}). Hence
\begin{align*}
|\f(\left<y,x\right>)|&=\left|\f\left(\sum_{j=1}^{+\infty}\eta_j^*\xi_j\right)\right|=\left|\sum_{j=1}^{+\infty}\f(\eta_j^*\eta_j)^{1/2}\f(\zeta_j^*\xi_j)\right|\le\\
    &\le\left(\sum_{j=1}^{+\infty}\f(\eta_j^*\eta_j)\right)^{1/2}\left(\sum_{j=1}^{+\infty}|\f(\zeta_j^*\xi_j)|^2\right)^{1/2}=\f(\left<y,y\right>)^{1/2}p_{\f,z}(x),
\end{align*}
finishing the proof.
\end{proof}

\begin{remark} The dual module of the module $\mathcal M$ is defined as the module of all $\mathcal A$-linear
and $\mathcal A$-valued bounded functionals. It is denoted by $\mathcal M'$. The module $\mathcal M$ always
can be embedded in $\mathcal M'$ via $\mathcal M\ni y\mapsto\Lambda_y\in\mathcal M'$,
$\Lambda_y(x)=\left<y,x\right>$. If this embedding is onto, the module $\mathcal M$ is called {\em
self-dual}.

It is well known that $l^2(\mathcal A)$ is not self-dual, except the algebra $\mathcal A$ is finite
dimensional. Namely, $l^2(\mathcal A)'$ can be described as the module of all sequences
$x=(\xi_1,\xi_2,\dots)$ such that the sequence of sums $\sum_{j=1}^n\xi_j^*\xi_j$ is norm bounded,
\cite[Proposition 2.5.5]{MT}.

Reading carefully the proof of the preceding proposition, one can see that nothing is changed if we replace
$l^2(\mathcal A)$ by $l^2(\mathcal A)'$. Indeed, the entire proof does not depend on the norm convergence of
the series $\sum_{j=1}^{+\infty}\xi_j^*\xi_j$.
\end{remark}

\begin{proposition} The unit ball in $l^2(\mathcal A)$ is not complete with respect to $\tau$. Its completion
is the unit ball in the dual module $l^2(\mathcal A)'$.
\end{proposition}

\begin{proof} First, we prove that the unit ball in $l^2(\mathcal A)$ is dense in the unit ball in
$l^2(\mathcal A)'$.

Let $x\in l^2(\mathcal A)'$, $x=(\xi_1,\xi_2,\dots)$. Since the sequence of sums $\sum_{j=1}^n\xi_j^*\xi_j$
is bounded it is convergent in strong (or weak, or ultraweak etc.) topology. By normality of $\f$ we have
$$\f\left(\sum_{j=1}^{+\infty}\xi_j^*\xi_j\right)=\sum_{j=1}^{+\infty}\f(\xi_j^*\xi_j),$$
implying that $\f(\sum_{j=n}^{+\infty}\xi_j^*\xi_j)\to0$, as $n\to+\infty$.

Thus, by the inequality (\ref{tau<tau_2}) $(\xi_1,\xi_2,\dots,\xi_n,0,0,\dots)\to x$ in each seminorm of the
form (\ref{seminorms}).

Next, we prove that $l^2(\mathcal A)'$ is complete. Let $x^\a=(\xi_1^\a,\xi_2^\a,\dots)$ be a Cauchy net in
the unit ball. Choosing an arbitrary normal state, and $\eta_k=1$, $\eta_j=0$ for $j\neq k$ we obtain that
$\xi_k^\a$ is a Cauchy net in weak-$*$ topology in the unit ball in $\mathcal A$. Hence, it is convergent,
say $\xi_k^\a\to\xi_k$ in the weak-$*$ topology.

Since multiplying is ultraweakly continuous, for any $n\in\N$ and for all $\eta_j$ which satisfy
(\ref{admissible_sequence}) we have
$$\sum_{j=1}^k|\f(\eta_j^*\xi_j^\a)|^2\to\sum_{j=1}^k|\f(\eta_j^*\xi_j)|^2.$$

Choosing $\eta_j=\xi_j/\f(\xi_j^*\xi_j)^{1/2}$ if $\f(\xi_j^*\xi_j)\neq0$ and $\eta_j=0$ otherwise, we get
$$\sum_{j=1}^k\f(\xi_j^*\xi_j)=\sum_{k=1}|\f(\eta_j^*\xi_j)|^2=\lim_\a\sum_{j=1}^k|\f(\eta_j^*\xi_j^\a)|^2\le||x||\le1.$$
Taking the limit as $k\to+\infty$ we conclude that $x=(\xi_1,\xi_2,\dots)\in l^2(\mathcal A)'$. To see that
$x$ is the limit of the Cauchy net $x^\a$ it is enough to take the limit over $\beta$ in
$$\sum_{j=1}^k|\f(\eta_j^*\xi_j^\a)-\f(\eta_j^*\xi_j^\beta)|^2\le\sum_{j=1}^{+\infty}|\f(\eta_j^*\xi_j^\a)-\f(\eta_j^*\xi_j^\beta)|^2<\e,$$
and finally the limit as $k\to+\infty$.
\end{proof}

Next, we want to study the restriction of $\tau$ to module $\mathcal A^n$ seen as a submodule of
$l^2(\mathcal A)$ consisting of those $x$ for which $\xi_j=0$ for all $j>n$.

\begin{proposition} a) On $\mathcal A^n$ the weak PF and our topology coincide, i.e.\ we have
$\tau_1|_{\mathcal A^n}=\tau|_{\mathcal A^n}$;

b) The embedding $i:\mathcal A^n\to l^2(\mathcal A)$, $i(\xi_1,\dots,\xi_n)=(\xi_1,\dots,\xi_n,0,\dots)$ is
continuous with respect to $(\tau|_{\mathcal A^n},\tau)$.
\end{proposition}

\begin{proof} a) We already have $\tau_1\subseteq\tau$. Let us prove the converse. An arbitrary seminorm of
the form (\ref{seminorms}) restricted to $\mathcal A^n$ has the form
\begin{equation}\label{seminorms_reduced}
p_{\f,y}(x)=\sqrt{\sum_{j=1}^n|\f(\eta_j^*\xi_j)|^2}.
\end{equation}
Consider the vectors $y_j=(0,\dots,0,\eta_j,0,\dots,0)$, where $\eta_j$ is the $j$-th entry. Then
$$p_{\f,y}(x)=\sqrt{\sum_{j=1}^n|\f(\left<y_j,x\right>)|^2}\le\sum_{j=1}^n|\f(\left<y_j,x\right>)|,$$
from which we conclude that $p_{\f,y}$ is continuous with respect to $\tau_1$;

b) One can easily check that
$$i^{-1}(\{x\;|\;p_{\f,\eta_1,\dots,\eta_n,\dots}(x)<\e\})=\{(\xi_1,\dots,\xi_n)\;|\;p_{\f,\eta_1,\dots,\eta_n}(\xi_1,\dots,\xi_n)<\e\}.$$
\end{proof}

\begin{proposition}\label{FiniteBallCompact}
The unit ball in $\mathcal A^n$ is compact with respect to $\tau|_{\mathcal A^n}$. Since $\mathcal A^n$ is
self-dual, the unit ball is also complete and hence totally bounded.
\end{proposition}

\begin{proof} In the case $n=1$, both topologies $\tau$ and $\tau_1$ are generated by seminorms
$\xi\mapsto|\f(\eta^*\xi)|$, $\eta\in\mathcal A$, $\f$ normal state. It is easy to verify that these
topologies are exactly the weak-$*$ topology on $\mathcal A$. Therefore, in this special case the conclusion
follows by Banach-Alaoglu theorem.

To obtain the result in general case, consider the product topology on $\mathcal A^n=\mathcal
A\times\dots\times\mathcal A$. Basic neighborhoods of zero have the form
$\{(\xi_1,\xi_2,\dots,\xi_n)\;|\;\forall j=1,2,\dots,n\:|\f_j(\eta_j^*\xi_j)|<\e_j\}$. Due to the
inequalities
$$\max_{1\le j\le n}|\f(\eta_j^*\xi_j)|\le\sqrt{\sum_{j=1}^n|\f(\eta_j^*\xi_j)|^2}\le\sqrt n\max_{1\le j\le
n}|\f(\eta_j^*\xi_j)|,$$
the topology $\tau$ is weaker then the product topology. Since the product of unit balls is compact in
stronger, product topology, and since $\tau$ is Hausdorff, we conclude that $\tau$ coincides with the product
topology on the product of unit balls.

Therefore, it remains to show that the unit ball in $\mathcal A^n$ is closed in the product of $n$ unit balls
in $\mathcal A$, i.e.\ that its complement is open.

Let $z=(\zeta_1,\dots,\zeta_n)\in\mathcal A^n$, $||z||>1$ be arbitrary. Let $\e>0$ be a number less then
$(||z||^2-||z||)/\sqrt n$, and let $\f$ be the normal state that attains its norm at
$\left<z,z\right>=\zeta_1^*\zeta_1+\dots+\zeta_n^*\zeta_n$ up to $\e\sqrt n$, i.e.\
$\f(\left<z,z\right>)>||z||^2-\e\sqrt n$. Consider the seminorm
$$p_{\f,z}(x)=\sqrt{|\f(\zeta_1^*\xi_1)|^2+\dots+|\f(\zeta_n^*\xi_n)|^2},\qquad x=(\xi_1,\xi_2,\dots,\xi_n).$$
We claim that the open set
$$G=\{x\;|\;p_{\f,z}(x-z)<\e\}$$
does not intersect the unit ball $B$.

Indeed, let $x\in G$. Then by classic Cauchy-Schwartz inequality we  have
\begin{align*}\e^2&>p_{\f,z}(x-z)^2=|\f(\zeta_1^*\xi_1)-\f(\zeta_1^*\zeta_1)|^2+\dots+|\f(\zeta_n^*\xi_n)-\f(\zeta_n^*\zeta_n)|^2\ge\\
                &\ge\frac1n|\f(\zeta_1^*\xi_1)+\dots+\f(\zeta_n^*\xi_n)-\f(\zeta_1^*\zeta_1)-\dots-\f(\zeta_n^*\zeta_n)|^2,
\end{align*}
or
\begin{align*}
\e\sqrt n&>|\f(\zeta_1^*\xi_1+\dots+\zeta_n^*\xi_n)-||z||^2+\e\sqrt n|\ge\\
        &\ge||z||^2-\e\sqrt n-|\f(\zeta_1^*\xi_1+\dots+\zeta_n^*\xi_n)|,
\end{align*}
i.e.
\begin{equation}\label{KSh1}
||z||^2-2\e\sqrt n<|\f(\zeta_1^*\xi_1+\dots+\zeta_n^*\xi_n)|=|\f(\left<z,x\right>)|.
\end{equation}
However, $\f(\left<z,x\right>)$ is a semi inner product and it satisfy Cauchy Schwartz inequality
\begin{equation}\label{KSh2}
|\f(\left<z,x\right>)|^2\le\f(\left<z,z\right>)\f(\left<x,x\right>)\le||z||^2||x||^2.
\end{equation}
From (\ref{KSh1}) and (\ref{KSh2}) we obtain
$$||z||\,||x||>||z||^2-2\e\sqrt n,$$
and taking into account how $\e$ is chosen, we have
$$||x||>\frac1{||z||}(||z||^2-2\e\sqrt n)>1.$$
Therefore, $x\notin B$, implying $B$ is a closed set. The proof is complete.
\end{proof}

\begin{proposition}The unit ball in $l^2(\mathcal A)$ is not totally bounded in $\tau$.
\end{proposition}

\begin{proof}Let $e_j=(0,\dots,0,1,0,\dots)$, where $1$ the unit of the algebra $\mathcal A$ stands at the
$j$-th entry. Let $\f$ be an arbitrary normal state and consider a seminorm $p=p_{\f,1,1,\dots}$ given by
$p(x)^2=\sum_{j=1}^{+\infty}|\f(\xi_j)|^2$.

We claim that the sequence $e_j$ is totally discrete in $p$. Indeed $p(e_i-e_j)^2=|\f(1)|^2+|\f(-1)|^2=2$,
i.e.\ $p(e_i-e_j)=\sqrt2$. Hence, the set $\{e_j\;|\;j\ge1\}$ is not totally bounded in $p$ and also in
$\tau$. The same is valid for a larger set - the unit ball.
\end{proof}

\section{"Compact" operators}

Let $y$, $z\in l^2(\mathcal A)$. The operator $l^2(\mathcal A)\to l^2(\mathcal A)$, $x\mapsto
z\left<y,x\right>$ is adjointable (its adjoint is $x\mapsto y\left<z,x\right>$) and bounded. The closed
linear hull of such operators is called the algebra of "compact" operators.

We say that the operator $T\in B^a(l^2(\mathcal A))$ is {\em compact} if its image of any (norm) bounded set
is a totally bounded set in topology $\tau$ described in the previous section. For the operator $T\in
B^a(l^2(\mathcal A))$ it is enough to maps the unit ball into a totally bounded set to be a compact operator.

\begin{remark} Totally bounded and relatively compact sets differ in general case (whenever the unit
ball is not complete). Also, throughout the literature, there is a certain ambiguity between terms {\em
completely continuous}, {\em compact} and {\em precompact} operators. Although it seems that terms {\em
completely continuous} and {\em precompact} are more accurate, we found that {\em compact} is more convenient
for our purpose.
\end{remark}

Before we prove that any "compact" operator is compact, we need a few lemmata.

\begin{lemma} For $S$, $T\subseteq l^2(\mathcal A)$ and a seminorm $p$ denote
$$d_p(S,T)=\sup_{x\in S}\inf_{y\in T}p(x-y)$$
(and note that $d_p$ is not symmetric). Let $S\subseteq l^2(\mathcal A)$. If for all seminorms $p$ of the
form (\ref{seminorms}) and all $\e>0$ there is a totally bounded set $S_{p,\e}$ such that
\begin{equation}\label{distance_seminorm}
d(S,S_{p,\e})<\e.
\end{equation}
then $S$ is also totally bounded.
\end{lemma}

\begin{proof} Denote
$$B_p(x;\e)=\{y\in l^2(\mathcal A)\;|\;p(x-y)<\e\}.$$
The condition (\ref{distance_seminorm}) gives
\begin{equation}\label{cover_p}
S\subseteq\bigcup_{x\in S_{p,\e}}B_p(x;\e/2),
\end{equation}
for all $\e>0$.

Let $\e>0$ be arbitrary. The set $S_{p,\e/2}$ is totally bounded in $p$ and hence there is a finite set
$\{c_1,\dots,c_m\}$ such that the union of balls $B_p(c_j;\e/2)$ covers $S_{\e/2}$. By (\ref{cover_p}) the
union of balls $B_p(c_j;\e)$ covers $S$.
\end{proof}

\begin{lemma}\label{CompletelyLimit}
Let $T_\a:l^2(\mathcal A)\to l^2(\mathcal A)$ be a net of compact operators such that $T_\a x\to Tx$ in
$\tau$ uniformly with respect to $||x||<1$. Then $T$ is also compact.
\end{lemma}

\begin{proof} For any $\e>0$ and any seminorm $p$ of the form (\ref{seminorms}) there is $\a$ such that
$\sup_{||x||<1}p(Tx-T_\a x)<\e$. Therefore,
$$d_p(T(B_{||\cdot||}(0;1)),T_\a(B_{||\cdot||}(0;1)))\le\e$$
and the conclusion follows from the previous Lemma.
\end{proof}

\begin{corollary}
Let $S\subseteq l^2(\mathcal A)$ be a set such that for all $\e>0$ there is a totally bounded (in $\tau$) set
$S_\e$ such that
\begin{equation}\label{distance}
d(S,S_\e)=\sup_{x\in S}\inf_{y\in S_\e}||x-y||<\e.
\end{equation}
Then $S$ is also totally bounded in $\tau$.

Also, let $T_n:l^2(\mathcal A)\to l^2(\mathcal A)$ be a sequence of compact operators that converges to $T$
in the operator norm. Then $T$ is also compact.
\end{corollary}

\begin{proof} Both conclusions follows from the fact that $\tau$ is coarser then the norm topology.
\end{proof}

\begin{lemma}\label{CompletelyLinear}
Let $T_1$ and $T_2$ be compact operators, and let $u_1$, $u_2\in\mathcal A$. Then $T_1u_1+T_2u_2$ is also
compact.
\end{lemma}

\begin{proof} Let $\e>0$ be arbitrary. Since $T_1$ and $T_2$ are compact there is a finite
$\e/2||u_1||$ net for $T_1(B_{||\cdot||}(0;1))$, say $c_1$, $c_2$, $\dots$, $c_n$, and a finite $\e/2||u_2||$
net for $T_2(B_{||\cdot||}(0;1))$, say $d_1$, $\dots$, $d_m$. Then the set $\{c_iu_1+d_ju_2\;|\;1\le i\le
i,1\le j\le m\}$ is a finite $\e$ net for $(T_1u_1+T_2u_2)(B_{||\cdot||}(0;1))$. Indeed, if $x\in
B_{||\cdot||}(0;1)$, then there is $i$ and $j$ such that $||T_1x-c_i||<\e/2||u_1||$ and
$||T_2x-d_j||<\e/2||u_2||$. Hence
$$||(T_1xu_1+T_2xu_2)-(c_iu_1+d_ju_2)||\le||T_1x-c_i||\,||u_1||+||T_2x-d_j||\,||u_2||<\e.$$
\end{proof}

\begin{theorem} Let $T:l^2(\mathcal A)\to l^2(\mathcal A)$ be a "compact" operator. Then $T$ is compact.
\end{theorem}

\begin{proof} In view od Lemmata \ref{CompletelyLimit} and \ref{CompletelyLinear}, it is enough to prove that
operators of the form $x\mapsto \Theta_{y,z}(x)=z\left<y,x\right>$ are compact.

In the special case, where $z=e_j\zeta$, $\zeta\in\mathcal A$ it immediately follows from Proposition
\ref{FiniteBallCompact}. Indeed, then $\Theta_{y,e_j\zeta}(B_{||\cdot||}(0;1))$ is contained in the ball of
radius $||T||$ in $\mathcal A^1$ which is totally bounded.

In general case, let $z=(\zeta_1,\zeta_2,\dots)$. Then $z=\sum_{j=1}^{+\infty}e_j\zeta_j$ where the series
converges in the norm. Since $||\Theta_{y,z}-\Theta_{y,z'}||\le||y||\,||z-z'||$, we have
$$\Theta_{y,z}=\lim_{n\to+\infty}\sum_{j=1}^n\Theta_{y,e_j\zeta_j}$$
and the required follows from the special case and Lemmata \ref{CompletelyLimit} and \ref{CompletelyLinear}.
\end{proof}

The converse is true in the special case where $\mathcal A=B(H)$ is the full algebra of all bounded linear
operators on a Hilbert space $H$. Before we prove such result we need a technical Lemma.

\begin{lemma}\label{ChooseState}
Let $\mathcal A=B(H)$ and let $a_j\in\mathcal A$, $j\ge 1$ be positive elements with $||a_j||>\delta$. There
is a normal state $\f$ and unitary elements $u_j$, $v_j\in \mathcal A$ such that $|\f(v_j^*a_ju_j)|>\delta$.
\end{lemma}

\begin{remark}Actually, we can choose $\f$ to be a vector state, and also we can choose $u_j=v_j$.
\end{remark}

\begin{proof} Let $\psi\in H$ be a unit vector, and let $\f$ be the corresponding vector state, i.e.\
$\f(a)=\left<a\psi,\psi\right>$. For all $a_j$ let $h_j$ be a unit vector such that
$\left<a_jh_j,h_j\right>>\delta$. As it is easy to see, there is a unitary $u_j$ such that $u_j\psi=h_j$.
Thus, we have $\f(u_j^*a_ju_j)=\left<u_j^*a_ju_j\psi,\psi\right>=\left<a_jh_j,h_j\right>>\delta$.
\end{proof}

\begin{theorem}\label{OtherDirection}
Let $\mathcal A=B(H)$ and let $T:l^2(\mathcal A)\to l^2(\mathcal A)$ be a compact operator. Then T is
"compact".
\end{theorem}

\begin{proof} Let $P_k$ denote the projection to the first $k$ coordinates, i.e.\
$P_k(\xi_1,\xi_2,\dots)\allowbreak=(\xi_1,\dots,\xi_k,0,0,\dots)$. It is well known that all $P_k$ are
"compact".

Suppose $T$ is not "compact". Then
$$\delta=\inf_{n\ge1}||(I-P_k)T||>0.$$
Indeed, otherwise either for some $k$ we have $(I-P_k)T=0$ and hence $T=P_kT$ is "compact", or there is a
sequence of positive integers $k_n$ such that $||T-P_{k_n}T||\to0$ from which it follows that $T$ is
"compact".

To simplify the calculations assume $||T||=1$. Then immediately, $\delta\le1$.

Define the sequence of projections $Q_n\in\{P_1,P_2,\dots\}$ and the sequences of vectors $x_n$, $y_n$ and
$z_n\in l^2(\mathcal A)$ in the following way. Let $Q_0=0$. If $Q_{n-1}$ is already defined, there is $x_n\in
l^2(\mathcal A)$ such that $||x_n||=1$ and $||(I-Q_{n-1})Tx_n||>\delta/2$. Denote $y_n=Tx_n$. Then, by
$||I-Q_{n-1}||=1$,
$$||y_n||\ge||(I-Q_{n-1})y_n||>\frac\delta2.$$
Since $\lim_{k\to+\infty}||(I-P_k)(I-Q_{n-1})y_n||=0$, there is a positive integer $k_n$ such that
$||(I-P_{k_n})(I-Q_{n-1})y_n||<\delta^2/8\le\delta/8$. Define $Q_n=P_{k_n}$ and
\begin{equation}\label{define_z}
z_n=Q_n(I-Q_{n-1})y_n.
\end{equation}

The sequences $y_n$ and $z_n$ have the following properties:

Firstly, by definition, there hold the inequalities
\begin{equation}\label{property1}
||(I-Q_n)(I-Q_{n-1})y_n||<\frac{\delta^2}8\le\frac\delta8,
\end{equation}
\begin{equation}\label{property1a}
||z_n||\le||y_n||\le||T||\,||x_n||=1,
\end{equation}
\begin{equation}\label{property1b}
||z_n||\ge||(I-Q_{n-1})y_n||-||(I-Q_n)(I-Q_{n-1})y_n||>\frac\delta2-\frac\delta8=\frac{3\delta}8.
\end{equation}

Secondly,
\begin{equation}\label{property2}
\left<z_n,y_n\right>=\left<z_n,z_n\right>.
\end{equation}
Indeed, since $z_n=Q_n(I-Q_{n-1})y_n$, we have
\begin{align*}
\left<z_n,y_n\right>&=\left<Q_n(I-Q_{n-1})y_n,y_n\right>=\\
                    &=\left<Q_n(I-Q_{n-1})y_n,(I-Q_{n-1})Q_ny_n\right>=\left<z_n,z_n\right>.
\end{align*}

Thirdly, for $m>n$ we have
\begin{equation}\label{property3}
||\left<z_m,y_n\right>||<\frac\delta8.
\end{equation}
Indeed, for such $m$ and $n$ we have $Q_{n-1}\le Q_n\le Q_{m-1}$, i.e.\ $I-Q_{m-1}\le I-Q_n\le I-Q_{n-1}$,
implying $I-Q_{m-1}=(I-Q_{m-1})(I-Q_n)(I-Q_{n-1})$, and thus
\begin{align*}
\left<z_m,y_n\right>&=\left<(I-Q_{m-1})z_m,y_n\right>=\\
    &=\left<z_m,(I-Q_{m-1})(I-Q_n)(I-Q_{n-1})y_n\right>=\\
    &=\left<z_m,(I-Q_n)(I-Q_{n-1})y_n\right>.
\end{align*}
Therefore, by (\ref{property1}) and (\ref{property1b})
$$||\left<z_m,y_n\right>||\le||z_m||\,||(I-Q_n)(I-Q_{n-1})y_n||\le\frac{\delta^2}8.$$

Let us construct a seminorm $p$, continuous in $\tau$, and a totally discrete sequence from
$T(\overline{B_{||\cdot||}(0;1)})$. Since by (\ref{property1b})
$||z_n||^2=||\upsilon^*_n\left<z_n,z_n\right>\nu_n||>(3\delta/8)^2$, we can choose $\f$ and $\upsilon_j$,
$\nu_j\in\mathcal A$ according to Lemma \ref{ChooseState}, such that
\begin{equation}\label{phi_choosed}
\f(\upsilon_n^*\left<z_n,z_n\right>\nu_n)>\frac{9\delta^2}{64}.
\end{equation}
Consider the seminorm $p$ given by
$$p(x)=\sqrt{\sum_{j=1}^{+\infty}|\f(\left<z_j\upsilon_j,x\right>)|^2}$$
By (\ref{define_z}) there is a sequence $\zeta_j\in\mathcal A$ such that
$$z_k=(0,\dots,0,\zeta_{k_{n-1}+1},\dots,\zeta_{k_n},0,\dots).$$
Define $\omega_n=\zeta_n\upsilon_n/\f(\upsilon_n^*\zeta^*_n\zeta_n\upsilon_n)^{1/2}$. Obviously
$\f(\omega_n^*\omega_n)=1$. Also, for $x=(\xi_1,\xi_2,\dots)$ we have
\begin{align*}
|\f(\left<z_n\upsilon_n,x\right>)|^2&=\left|\sum_{j=k_{n-1}+1}^{k_n}\f(\upsilon_n^*\zeta^*_j\zeta_j\upsilon_n)^{1/2}\f(\omega_j^*\xi_j)\right|^2\le\\
    &\le\sum_{j=k_{n-1}+1}^{k_n}\f(\upsilon_n\zeta^*_j\zeta_j\upsilon_n)\sum_{j=k_{n-1}+1}^{k_n}|\f(\omega_j^*\xi_j)|^2=\\
    &=\f(\upsilon^*_n\left<z_n,z_n\right>\upsilon_n)\sum_{j=k_{n-1}+1}^{k_n}|\f(\omega_j^*\xi_j)|^2
\end{align*}
Including (\ref{property1a}) we obtain
$\f(\upsilon^*_n\left<z_n,z_n\right>\upsilon_n)\le||\upsilon^*_n\left<z_n,z_n\right>\upsilon_n||=||z_n||^2\le1$
and hence
$$p(x)^2=\sum_{n=1}^{+\infty}|\f(\left<z_n,x\right>)|^2\le\sum_{j=1}^{+\infty}|\f(\omega_j^*\xi_j)|^2=
    p_{\f,\omega_1,\dots,\omega_n,\dots}(x)^2.$$
Thus, we conclude that $p$ is well defined and also that it is continuous with respect to $\tau$.

Also, $||x_n\nu_n||=||x_n||$, i.e.\ $y_n\nu_n=Tx_n\nu_n\in T(\overline{B(0;1)})$. Finally we shall prove that
$y_n\nu_n$ is a totally discrete sequence. Indeed, for $m>n$ we have
\begin{align*}
p(y_m\nu_m-y_n\nu_n)&\ge|\f(\left<z_m\upsilon_m,y_m\nu_m-y_n\nu_n\right>)|\ge\\
    &\ge|\f(\upsilon_m^*\left<z_m,y_m\right>\nu_m)|-|\f(\upsilon_m^*\left<z_m,y_n\right>\nu_n)|.
\end{align*}
However, by (\ref{property2}) and (\ref{phi_choosed}),
$$|\f(\upsilon_m\left<z_m,z_m\right>\nu_m)|>\frac{9\delta^2}{64}$$
and, by (\ref{property3})
$$|\f(\upsilon_m^*\left<z_m,y_n\right>\nu_n)|\le||\left<z_m,y_n\right>||<\frac{\delta^2}8.$$
Therefore
$$p(y_m\nu_m-y_n\nu_n)>\frac{9\delta^2}{64}-\frac{\delta^2}8=\frac{\delta^2}{64}.$$
\end{proof}

\section{An example and a comment}

The proof of Theorem \ref{OtherDirection} depends on Lemma \ref{ChooseState}. Hence it is valid for all
unital $W^*$-algebras that satisfy the mentioned Lemma. We do not know how to describe such algebras, but it
should be mentioned that Lemma \ref{ChooseState} does not hold for infinite dimensional commutative
$W^*$-algebras.

\begin{example}
In any infinite dimensional commutative $W^*$ algebra $\mathcal A$, there is a sequence $p_j$ of nontrivial
mutually orthogonal projections. Since $\sum_{j=1}^np_j$ is an increasing sequence,
$p=\sum_{j=1}^{+\infty}p_j\in\mathcal A$. Therefore, for an arbitrary normal state $\f$ the series
$\sum_{j=1}^{+\infty}\f(p_j)$ is convergent. The algebra is commutative, and for all unitary $\upsilon_j$,
$\nu_j$ we have
$$|\f(\upsilon_jp_j\nu_j)|=|\f(p_j\upsilon_j\nu_j)|\le\f(p_j)^{1/2}\f(\nu_j^*\upsilon_j^*\upsilon_j\nu_j)^{1/2}\to0.$$
Thus, Lemma \ref{ChooseState} is not valid for commutative $W^*$ algebras.

Moreover, we can use this sequence of projection to construct an operator which is compact but is not
"compact". Indeed, let $T:l^2(\mathcal A)\to l^2(\mathcal A)$ be the operator defined by
\begin{equation}\label{counterexample}
Tx=T(\xi_1,\xi_2,\dots)=(p_1\xi_1,p_2\xi_2,\dots).
\end{equation}
Then, $T$ is not "compact". Indeed, if it is "compact", for all $\e>0$ there is an operator $S$ of the form
$S=\sum_{j=1}^n\lambda_j\Theta_{y_j,z_j}$ such that $||T-S||<\e/3$. Since $P_kz_j-z_j\to0$, as $k\to+\infty$
for all $1\le j\le n$ implies $||P_kS-S||\to0$, there is $k$ large enough such that $||P_kS-S||<\e/3$ and
then $||T-P_kT||\le||T-S||+||S-P_kS||+||P_k(S-T)||<\e$. However, as it is easy to see
$||T-P_kT||\ge||Te_k-P_kTe_k||=||1-p_k||=1$.

On the other hand, for an arbitrary semi norm of the form (\ref{seminorms}) we have $p((T-P_kT)x)\to0$,
uniformly with respect to $x\in B_{||\cdot||}(0,1)$. Indeed, $\mathcal A$ is commutative and therefore
$\xi_j^*\xi_j\eta_j^*\eta_j\le||\xi_j||^2\eta_j^*\eta_j$, and further
$\f(\xi_j^*\xi_j\eta_j^*\eta_j)\le||\xi_j||^2\sup_j\f(\eta_j^*\eta_j)\le 1$, by $||x||<1$ and
(\ref{admissible_sequence}). Thus, we have
$$p((T-P_kT)x)^2=\sum_{j=k+1}^{+\infty}|\f(\eta_j^*p_j\xi_j)|^2\le\sum_{j>k}\f(p_j)\f(\xi_j^*\xi_j\eta_j^*\eta_j)\le
    \sum_{j>k}\f(p_j)\to0.$$
Hence, $T$ is compact by Lemma \ref{CompletelyLimit}
%
%On the other hand, the operator $T$ given by (\ref{counterexample}) as well, but acting on the dual module
%$l^2(\mathcal A)'$ is "compact". Indeed, $z=(p_1,p_2,\dots)\in l^2(\mathcal A)'$ since $||\sum_1^np_j||\le1$
%and hence $Tx=... Djokson
\end{example}

\begin{remark} Topology $\tau$ defined in this note highly depends on coordinates, and therefore it is
inappropriate for Hilbert modules other then $l^2(\mathcal A)$. One might try to define a topology by semi
norms
\begin{equation}\label{seminormsmore}
p_{\f,z_j}(x)=\sqrt{\sum_{j=1}^{+\infty}|\f(\left<z_j,x\right>)|^2},
\end{equation}
where $\f$ is a normal state, and $z_j$ is an orthogonal sequence, that satisfies
$\sup_{j\ge1}\f(\left<z_j,z_j\right>)=1$. These semi norms are generalization of those given by
(\ref{seminorms}). Indeed, semi norm (\ref{seminormsmore}) become semi norms (\ref{seminorms}) by choosing
$z_j=e_j\eta_j$.

However, such new topology is in the case of $l^2(\mathcal A)$ larger then $\tau$, even if we suppose that
$z_j$ is even more orthonormal. Namely, if $\mathcal A=B(H)$, $H$ infinite dimensional, there is a Cuntz
$\infty$-tuple, i.e.\ a sequence of isometries $v_j$ satisfying $v_j^*v_j=1$ and
$\sum_{j=1}^{+\infty}v_jv_j^*=1$. Then, it is easy to see that $x_j=(v_j,0,0,\dots)$ is orthonormal. But, in
the semi norm $p_{\f,x_j}$ of the form (\ref{seminormsmore}) the sequence $x_j$ itself is totally discrete.
\end{remark}

\subsection*{Acknowledgement}
The authors was supported in part by the Ministry of education and science, Republic of Serbia, Grant
\#174034.

\end{document}